\documentclass[reqno]{amsart}

\usepackage{amscd,amsthm,amsmath,amssymb,mathtools}
\usepackage{upgreek,bm}
\usepackage{verbatim}
\usepackage{xcolor}
\usepackage{url}
\usepackage{enumerate,enumitem}
\usepackage{graphicx}

\usepackage{algorithm}
\usepackage{algorithmic}
\usepackage{epstopdf,epsfig,subfigure}
\usepackage{curve2e}
\usepackage{mathrsfs}
\usepackage{array}
\usepackage{stackrel}

\usepackage[numbers,sort&compress]{natbib}

\usepackage{setspace}
\usepackage[
    bookmarks=true,         
    unicode=false,          
    pdftoolbar=true,        
    pdfmenubar=true,        
    pdffitwindow=false,     
    pdfstartview={FitH},    
    pdftitle={NSE with dynamic boundary conditions},    
    pdfauthor={Author},     
    pdfsubject={Subject},   
    pdfcreator={Creator},   
    pdfproducer={Producer}, 
    pdfkeywords={keyword1, key2, key3}, 
    pdfnewwindow=true,      
    colorlinks=true,       
    linkcolor=red,          
    citecolor=blue,        
    filecolor=magenta,      
    urlcolor=cyan,           
hypertexnames=false
]{hyperref}

\theoremstyle{plain}

\newtheorem{theorem}{Theorem}[section]

\newtheorem{lemma}[theorem]{Lemma}
\newtheorem{corollary}[theorem]{Corollary}

\newtheorem{assumption}[theorem]{Assumption}

\theoremstyle{definition}

\newtheorem{remark}[theorem]{Remark}

\numberwithin{equation}{section}

\usepackage{geometry}

\newcommand{\linspan}{\mathop{\rm span}\nolimits}

\newcommand{\supp}{\mathop{\rm supp}\nolimits}
\newcommand{\diver}{\mathop{\rm div}\nolimits}
\newcommand{\curl}{\mathop{\rm curl}\nolimits}

\newcommand{\rest}{\left.\kern-2\nulldelimiterspace\right|_}
\newcommand{\norm}[2]{\left|#1\right|_{#2}}
\newcommand{\dnorm}[2]{\left\|#1\right\|_{#2}}

\newcommand{\Id}{{\mathbf1}}
\newcommand{\indf}{1}

\newcommand{\clE}{{\mathcal E}}

\newcommand{\clL}{{\mathcal L}}

\newcommand{\clN}{{\mathcal N}}
\newcommand{\clO}{{\mathcal O}}

\newcommand{\clU}{{\mathcal U}}
\newcommand{\clV}{{\mathcal V}}
\newcommand{\clW}{{\mathcal W}}

\newcommand{\clY}{{\mathcal Y}}

\newcommand{\bbA}{{\mathbb A}}
\newcommand{\bbB}{{\mathbb B}}

\newcommand{\bbD}{{\mathbb D}}

\newcommand{\bbN}{{\mathbb N}}

\newcommand{\bbR}{{\mathbb R}}
\newcommand{\bbS}{{\mathbb S}}

\newcommand{\bfD}{{\mathbf D}}

\newcommand{\bfH}{{\mathbf H}}

\newcommand{\bfP}{{\mathbf P}}

\newcommand{\bfT}{{\mathbf T}}

\newcommand{\bfV}{{\mathbf V}}

\newcommand{\bfY}{{\mathbf Y}}

\newcommand{\fkA}{{\mathfrak A}}
\newcommand{\fkB}{{\mathfrak B}}

\newcommand{\fkD}{{\mathfrak D}}

\newcommand{\fkL}{{\mathfrak L}}

\newcommand{\fkP}{{\mathfrak P}}

\newcommand{\fkV}{{\mathfrak V}}

\newcommand{\rmD}{{\mathrm D}}

\newcommand{\rmS}{{\mathrm S}}

\newcommand{\bff}{{\mathbf f}}
\newcommand{\bfg}{{\mathbf g}}

\newcommand{\bfn}{{\mathbf n}}

\newcommand{\bft}{{\mathbf t}}

\newcommand{\bfv}{{\mathbf v}}
\newcommand{\bfw}{{\mathbf w}}

\newcommand{\bfy}{{\mathbf y}}

\newcommand{\rmc}{{\mathrm c}}
\newcommand{\rmd}{{\mathrm d}}
\newcommand{\rme}{{\mathrm e}}

\newcommand{\fka}{{\mathfrak a}}

\definecolor{DarkBlue}{rgb}{0,0.08,0.45}
\definecolor{DarkRed}{rgb}{.65,0,0}
\definecolor{applegreen}{rgb}{0.55, 0.71, 0.0}

\newcounter{mymac@matlab}
\newcommand{\matlab}{MATLAB%
   \ifnum\value{mymac@matlab}<1%
   \textregistered%
   \setcounter{mymac@matlab}{1}%
   \fi%
  }

\providecommand{\argmax}{\operatorname*{argmax}}

\pdfoutput=1

\begin{document}
\title{Stabilization of nonautonomous Navier--Stokes flows under dynamic slip boundary conditions}
\author{Buddhika Priyasad$^{\tt1}$}
\author{S\'ergio S.~Rodrigues$^{\tt2}$}
\thanks{
\vspace{-1em}\newline\noindent
\newline\noindent
  $^{\tt1}$ Department of Mathematics and Statistics, University of Konstanz, 78457 Konstanz, Germany.
 \newline\noindent  
 $^{\tt2}$ Center for Mathematics and Applications (NOVA Math) and Department of Mathematics, NOVA School of Science and Technology (NOVA FCT), NOVA University of Lisbon, Campus de Caparica
2829-516,  Portugal.
\quad
\newline\noindent
{\sc Emails}:
{\small\tt priyasad@uni-konstanz.de,\quad ssi.rodrigues@fct.unl.pt}%
}

\begin{abstract}
Exponential stabilizability of the incompressible Navier–Stokes equations under dynamic slip boundary conditions toward arbitrary time-dependent trajectories is proven. The feedback control law is constructed explicitly using oblique projections and realized through a finite number of spatially localized interior actuators, without requiring spectral assumptions. The approach extends to various slip boundary condition types (Navier, vorticity-type, and Neumann) and applies to multi-connected domains. Weak solution existence and exponential decay estimates are established, with the stabilization rate depending on the boundary dynamics parameters.
\end{abstract}
\maketitle
\pagestyle{myheadings} \thispagestyle{plain} \markboth{\sc  B. Priyasad and S. S.
Rodrigues}{\sc }
	\section{Introduction}	
We consider the stabilizability of the incompressible Navier–Stokes system with dynamic slip boundary conditions on a smooth bounded domain $\Omega \subset \mathbb{R}^d$, with $d \in\{2,3\}$, towards a reference solution~$\widehat\bfy$ satisfying, for time~$t>0$,
\begin{subequations}\label{nse-haty-intro}
\begin{align}
&\tfrac{\partial}{\partial t}\widehat \bfy =\nu\Delta\widehat\bfy -\langle\widehat\bfy \cdot\nabla\rangle \widehat\bfy  + \nabla \widehat p + \bff,\qquad \diver \widehat\bfy = 0, \label{nse-haty-intro-dyn}\\
&\beta\tfrac{\partial}{\partial t}\widehat\bfy\rest{\partial\Omega}=-  \fkD\widehat\bfy\rest{\partial\Omega}-\nu\fkL\widehat\bfy+ \bfg,\qquad (\widehat\bfy  \cdot \bfn)\rest{\partial\Omega} = 0,\label{nse-haty-intro-bdry}\\
&\widehat\bfy(\cdot,0) = \widehat\bfy_0,\qquad\widehat\bfy(\cdot,0)\rest{\partial\Omega}=\widehat\bfy_{0\partial}. \label{nse-haty-intro-ic}
\end{align}
Here~$\widehat \bfy(x,t)\in\bbR^d$ represents the velocity of the fluid and~$\widehat p(x,t)\in\bbR$ is the pressure, $x\in\Omega$. The dynamic boundary conditions, in~\eqref{nse-haty-intro-bdry}, includes the usual boundary tangency of the velocity field at the boundary~$\partial\Omega$ of~$\Omega$, $(\widehat\bfy  \cdot \bfn)\rest{\partial\Omega}=0$,  where~$\bfn=\bfn(\overline x)$ stands for the unit outward normal vector at~$\overline x\in\partial\Omega$; further, it includes a condition on the tangential component of~$\widehat\bfy\rest{\partial\Omega}$, where the operator~$\fkL$ satisfies the Green formula
\begin{equation}\label{barL1}
(\Delta w,z)_{L^2(\Omega,\bbR^d)}=-(\overline \fkL w,\overline \fkL z)_{L^2(\Omega,\bbR^{d_1})}+(\fkL w,z)_{L^2(\partial\Omega,\bfT\partial\Omega)},
\end{equation}
for divergence-free smooth vector fields~$w$ and~$z$ tangent to the boundary and for some
\begin{equation}\label{barL2}
\overline \fkL\colon W^{1,2}(\Omega,\bbR^d) \to L^2(\Omega,\bbR^{d_1}).
\end{equation}
In~\eqref{barL1}, $\bfT\partial\Omega$ denotes the tangent bundle of~$\partial\Omega$. The operator~$\fkD\colon\fkV\to\fkV'$ is symmetric and positive definite, for some space~$\fkV\xhookrightarrow{\rmd}L^2(\partial\Omega,\bfT\partial\Omega)$, with
\begin{equation}\label{fkDalpha}
\alpha\coloneqq\min\limits_{\xi\in\fkV}\tfrac{\langle\fkD\xi,\xi\rangle_{\fkV',\fkV}}{\norm{\xi}{L^2(\partial\Omega,\bfT\partial\Omega)}^2}>0.
\end{equation}
\end{subequations}
The parameters~$\nu>0$ and~$\beta>0$ are constants. Thus, we will have dynamic slip boundary conditions. Note that~$\beta=0$ corresponds to static slip boundary conditions. We will cover the dynamic version of the classical Navier boundary conditions, with~$\fkL w=[(\bfD w)\bfn]_\bfT$ and~$\overline\fkL w=\sqrt2(\nabla w+(\nabla w)^\top)\in\bbR^{d\times d}\sim\bbR^{d^2}$, with~$z_\bfT\coloneqq z-(n\cdot z)\bfn$ and~$\nabla w=[\tfrac{\partial w_i}{\partial x_j}]$;   dynamic vorticity-type boundary conditions, with~$\fkL w=-\bfn\times\curl w$ in case~$d=3$ and~$\fkL w=(\curl w)\bfn^\perp=(\curl w)(-\bfn_2,\bfn_1)$ in case~$d=2$,  and~$\overline\fkL w=\curl w\in\bbR^{2d-3}$; and dynamic Neumann boundary conditions, with~$\fkL w= [(\nabla w)\bfn]_\bfT$ and~$\overline\fkL w=\nabla w\in\bbR^{d^2}$. The operator~$\fkD$ can be, for example, a scaled identity~$\fkD=\alpha\Id$, $\alpha>0$, or a shifted (nonnegative definite) Laplace-Beltrami operator~$\Delta_\partial$ as~$\fkD=\alpha\Id+\Delta_\partial$.  The spatial domain~$\Omega$ is multi-connected with boundary consisting of a finite number of connected components~\cite[Appen.~I]{Temam01}.
The terminology for the boundary conditions may vary within the literature: curl-type boundary conditions are sometimes called Lions boundary conditions~\cite{Kelliher06,PhanRod17} and Neumann boundary conditions~\cite{MitreaMonniaux09b}. In~\eqref{nse-haty-intro}, $\bff, \bfg$ denote external forces depending on both space and time, with $\bff: \bbR^d \times \bbR \to \bbR^d$ and $\bfg: \bbR^d \times \bbR \to \bbR^d$. The term $\bff$ represents volume (interior) forces, while $\bfg$ acts on the boundary.
\subsection{Stabilizability to a reference trajectory}
We assume that the reference trajectory~$\widehat\bfy$ has a desired asymptotic behavior, which we would like to track. Due to instability, the trajectories~$\bfy$ issued from a different initial state may have a different asymptotic behavior or may even blow up in finite time, in some norm of interest. 
In this situation we need and seek a control forcing to track the target~$\widehat\bfy$. The control dynamics reads 
\begin{subequations}\label{nse-y-intro}
\begin{align}
&\tfrac{\partial}{\partial t} \bfy =\nu\Delta\bfy -\langle\bfy \cdot\nabla\rangle \bfy  + \nabla p +\bff+\sum_{i=1}^mu_i\Phi_i,\qquad \diver \bfy = 0, \label{nse-y-intro-dyn}\\
&\beta\tfrac{\partial}{\partial t}\bfy=-  \fkD\bfy-\nu\fkL\bfy+ \bfg,\qquad (\bfy  \cdot \bfn)\rest{\partial\Omega} = 0,\label{nse-y-intro-bdry}\\
&\bfy(\cdot,0) =\bfy_0,\qquad\bfy(\cdot,0)\rest{\partial\Omega}=\bfy_{0\partial} \label{nse-y-intro-ic}
\end{align}
\end{subequations}
with the control forcing given by a linear combination of a finite set of actuators~$\Phi_i=\Phi_i(x)$. The vector~$u=u(t)\in\bbR^m$ is our control input that we use to tune the actuators.

Further, motivated by robustness properties, we seek the input in feedback form~$u(t)=K(\bfy(\cdot,t) -\widehat\bfy(\cdot,t) )$ for a suitable operator~$K$. In other words, our goal is  to find the feedback-input operator~$K$ such that with the input~$u=K(\bfy -\widehat\bfy)$, we will have
\begin{equation}\label{goal-intro}
\lim_{t\to\infty}\norm{(\bfy(\cdot,t),\bfy(\cdot,t)\rest{\partial\Omega}) -(\widehat\bfy(\cdot,t),\widehat\bfy(\cdot,t)\rest{\partial\Omega}) }{X}=0,
\end{equation}
in a suitable normed space~$X$. In fact, we will be seeking exponential convergence, where
\begin{subequations}\label{goal-intro-exp}
\begin{equation}
\norm{\bfY(\cdot,t) -\widehat\bfY(\cdot,t) }{X}\le C\rme^{-\mu(t-s)}\norm{\bfY(\cdot,s) -\widehat\bfY(\cdot,s) }{X},\quad\mbox{for all }t\ge s\ge 0,
\end{equation}
for suitable constants~$C\ge1$ and~$\mu>0$,
where
\begin{equation}
\bfY(\cdot,t)\coloneqq (\bfy(\cdot,t),\bfy(\cdot,t)\rest{\partial\Omega})\quad\text{and}\quad \widehat\bfY(\cdot,t)\coloneqq (\widehat\bfy(\cdot,t),\widehat\bfy(\cdot,t)\rest{\partial\Omega}).
\end{equation}
\end{subequations}

\subsection{On the literature}
Stabilization of evolving processes is an important task in control applications. This manuscript is concerned with the stabilizability of Navier--Stokes modeling the velocity of a fluid. We focus on models subject to dynamic boundary conditions, the stabilization of which seems to be not investigated in the literature.

Dynamic boundary conditions have been the subject of several works in recent years as, for example,~\cite{Maringova19,PrazakPryiasad24,PrazakZelina23,PrazakZelina24,AbbatielloBulicMaring21,RackeZheng03} focusing of models of fluid dynamics.

Further works for parabolic-like models are~\cite{DenkPrussZacher08,GrobbVanDalsenSauer89,Hintermann89,Guidetti16,HombergKrumbRehberg13,MiranvilleZelik05}. Dynamic boundary conditions can also be implicitly present under static Wentzell boundary conditions~\cite{DenkKunzePloss21}.

We consider Navier--Stokes equations under dynamic slip boundary conditions. The static version of these conditions have been the subject of several works in recent years, for example,
\cite{CloMikRob98,MitreaMonniaux09a,MitreaMonniaux09b,AmroucheSeloula11,Kelliher06,ChemetovCiprianoGavrilyuk10} and 
\cite{AmroucheRejaiba14,AcTapiaAmrouConcaGhosh21,KelliherLacLFilhoNLopes25,XiaoXin13,PhanRod17}

Since fluids are ubiquitous in real world, the stabilizability of Navier--Stokes flows has received the dedication of many researchers, under static boundary conditions, for example, we can mention the earlier works with~\cite{Barbu03,BarbuTri04,BarbuLasTri06,Fursikov01,Fursikov04,Lefter09,Munteanu12,Munteanu3D12,Raymond19,LasieckaPriyTrig21a,LasieckaPriyTrig21b,lasiecka2020uniform,Rod21-amo}
and the more recent
\cite{AzouaniTiti14,BreitenKunisch20,BanschBennerSaakWech15,Barbu_TAC13,DharRaymThev11,FurKornev12,BadTakah11,Munteanu12,Munteanu3D12,Raymond19,LasieckaPriyTrig21a,LasieckaPriyTrig21b,Lasiecka2024, BarRodShi11,Rod21-amo}.
The stabilizability to a steady state is often considered and proven making use of spectral properties of the linearization of the dynamics around the steady state.

In this manuscript we consider the stabilizability to a more general, possibly time-dependent trajectory, where the spectral properties are not an appropriate tool to investigate stability properties~\cite{Wu74}. The mathematical study of the stabilizability to time-dependent trajectories has been initiated in~\cite{BarRodShi11} with internal controls, then several works for several models, here we mention~\cite{Azmi22,Rod21-amo} focused on Navier--Stokes equations. 
We shall develop a variant of the strategy introduced in~\cite{Rod21-aut}, applied in~\cite{RodSeifu24} to Navier--Stokes equations under static curl-type boundary conditions, to the case of dynamic slip boundary conditions.

\subsection{Contents}
The rest of this paper is organized as follows. In Section~\ref{sec:Stabil-statemResults} we present a literature background on stabilization of fluids and present the novel contributions of the paper. The technical mathematical setting is given in Section~\ref{sec:mathSett} and,  subsequently, the existence of weak solutions is proven in Section~\ref{sec:well-posed} for both the free and the controlled dynamics. 
Section~\ref{sec:monotStokes} is dedicated to showing that by placing appropriately a large enough number of actuators we can impose a suitable monotonicity property of the sum between the Stokes operator and the feedback-input operator. The main stabilizability result is proven in Section~\ref{sec:stabil} under suitable general assumptions required throughout the paper. The satisfiability of these assumptions is shown in Section~\ref{sec:satisfAss}.


\section{Stabilization of Navier--Stokes equations with dynamic boundary conditions}\label{sec:Stabil-statemResults}

Stabilization of evolving processes is a fundamental challenge in control theory and its applications to Science and Engineering. This manuscript is focused on the stabilizability of
Navier--Stokes equations modeling the velocity field of a fluid, specifically on models subject to dynamic boundary conditions, a class of conditions for which the stabilization problem has received limited attention in the existing literature despite its practical importance.

\subsection{Dynamic boundary conditions in fluid dynamics}

Dynamic boundary conditions represent a significant complement to more classical static boundary conditions, allowing for complex interactions between fluid flow and dynamics at the boundary. Unlike static boundary conditions that prescribe fixed constraints at domain boundaries, dynamic boundary conditions evolve according to differential equations, thereby introducing additional degrees of freedom and complexity to the control problem.

In recent years, dynamic boundary conditions have been the subject of substantial investigation. Notable works on parabolic-like models include those by Denk, Pr\"{u}ss, and Zacher~\cite{DenkPrussZacher08}, as well as contributions by Grobb, Van Dalsen, and Sauer~\cite{GrobbVanDalsenSauer89}, Hintermann~\cite{Hintermann89}, Guidetti~\cite{Guidetti16}, Homberg, Krumb, and Rehberg~\cite{HombergKrumbRehberg13}, and Miranville and Zelik~\cite{MiranvilleZelik05}. For fluid dynamics applications specifically, dynamic boundary conditions have been extensively studied in~\cite{Maringova19,PrazakPryiasad24,PrazakZelina23,PrazakZelina24,AbbatielloBulicMaring21,RackeZheng03, Pruss2006}. Furthermore, dynamic boundary conditions can arise implicitly through static Wentzell boundary conditions, as discussed in~\cite{DenkKunzePloss21}.

\subsection{Background and motivation}

The present work considers Navier--Stokes equations under dynamic slip boundary conditions. The investigation of dynamic conditions introduces new mathematical challenges in comparison to static boundary conditions, namely due to the coupling between the flow dynamics and boundary evolution, necessitating novel analytical techniques and control strategies.

The importance of slip boundary conditions in modeling and applications has led to  works dedicated by several authors to the investigation of existence and regularity of solutions; we refer the reader to~\cite{CloMikRob98,MitreaMonniaux09a,MitreaMonniaux09b,AmroucheSeloula11,Kelliher06,ChemetovCiprianoGavrilyuk10,AmroucheRejaiba14,AcTapiaAmrouConcaGhosh21,KelliherLacLFilhoNLopes25,XiaoXin13,PhanRod17} and references therein.

 \subsection{Stabilization of Navier--Stokes flows}

Since fluids are ubiquitous in practical applications, the stabilizability of Navier--Stokes flows has received extensive attention from the control theory community. Under static boundary conditions, numerous works have contributed to our understanding of stabilization mechanisms. Early foundational contributions include~\cite{Barbu03,BarbuTri04,BarbuLasTri06,Fursikov01,Fursikov04,Lefter09,Munteanu12,Munteanu3D12,Raymond19,LasieckaPriyTrig21a,LasieckaPriyTrig21b,BarRodShi11,Rod21-amo}, while more recent developments can be found in~\cite{AzouaniTiti14,BreitenKunisch20,BanschBennerSaakWech15,Barbu_TAC13,DharRaymThev11,FurKornev12,BadTakah11}.

A significant portion of this research has focused on stabilization to steady states, where classical spectral analysis of linearized dynamics provides effective tools for designing controllers. However, this manuscript pursues the more ambitious goal of stabilization to general trajectories that are allowed to be time-dependent. This generalization is motivated by practical scenarios where equilibrium states either do not exist or are inappropriate for the application and, consequently, tracking time-varying reference trajectories becomes essential. In such settings, spectral methods are fundamentally inadequate~\cite{Wu74}, and alternative analytical frameworks have been sought during recent years.

\subsection{Novel contributions and technical approach}

The mathematical theoretical study of stabilization to time-dependent trajectories has been initiated in~\cite{BarRodShi11} with internal controls for Navier--Stokes equations. Subsequent works have extended these ideas to various settings~\cite{Azmi22,Rod21-amo}, providing methodologies applicable to different boundary conditions and control configurations.

In this manuscript, we develop a variant of the stabilization strategy introduced in~\cite{Rod21-aut} and successfully applied to Navier--Stokes equations under static curl-type boundary conditions in~\cite{RodSeifu24}. The key contribution of this work is the extension of this approach to the more complex setting of dynamic slip boundary conditions. This extension is nontrivial, as dynamic boundary conditions introduce additional coupling effects and require careful treatment of the estimates that underpin stabilization analysis. We demonstrate how to overcome these obstacles and establish sufficient conditions for stabilizability of this new class of problems.

\subsection{Main result}
We state the main result of the manuscript.

\subsubsection{The actuators}
To state our main stabilizability result, firstly it is necessary to specify the set of actuators~$\Phi_i$. Throughout this work, up to a translation, a rotation, and a homothety, each actuator~$\Phi_i\in L^2(\Omega,\bbR^d)$ is of the form
\begin{subequations}\label{Act-intro}
\begin{align}
&\Phi_i\in\{\Phi^{[2]}_1,\Phi^{[2]}_2\},\qquad&&\Phi^{[2]}_1\coloneqq \begin{bmatrix}
    \indf_{\omega}\\0
\end{bmatrix},\quad\Phi^{[2]}_1\coloneqq\begin{bmatrix}
    0\\\indf_{\omega}
\end{bmatrix},
\intertext{in case $d=2$, and}
&\Phi_i\in\{\Phi^{[3]}_1,\Phi^{[3]}_2,\Phi^{[3]}_3\},\qquad&&\Phi^{[3]}_1\coloneqq\begin{bmatrix}
    \indf_{\omega}\\0\\0
\end{bmatrix},\quad\Phi^{[3]}_2\coloneqq\begin{bmatrix}
    0\\\indf_{\omega}\\0
\end{bmatrix},\quad\Phi^{[3]}_3\coloneqq\begin{bmatrix}
    0\\0\\\indf_{\omega}
\end{bmatrix},
\end{align}
in case $d=3$, where~$\indf_{\omega}$ is the indicator function of a suitable subset~$\omega\subset\Omega$.
\end{subequations}
Let us denote the number of elements in a partition~$\bfP$ of~$\Omega$ by $\#\bfP$. For a disk-domain~$\Omega$,  the trivial case where $\#\bfP = 1$ (thus, where the partition is the singleton  $\bfP=\bfP_1=\{\Omega\}$) is depicted in Fig.~1; in this case we have chosen a square-shaped support~$\omega$ for the actuators centered at the center of the disk. For the cases with $\#\bfP > 1$, the placement of actuators will follow a strategy as illustrated again in Fig.~1: one (rescaled) actuator support in each element of the partition. This approach generalizes to more general domains that admit  decompositions into small subdomains, each possessing a certain uniform Sobolev extension property. The precise statement of this property is given in Assumption~5.1.

\begin{figure}[htbp]%
    \centering%
   \subfigure[Placement in a disk.] 
    {\includegraphics[width=.85\textwidth]{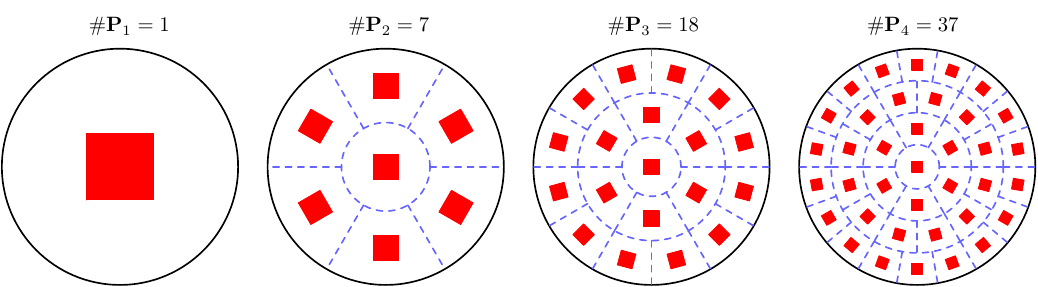}} 
  \caption{Placement  of square-shaped supports of the actuators~$\omega_i$ in a disk.}%
     \label{fig:act}%
\end{figure}

{ More precisely, consider the set of $m$ actuators where $m = d \cdot \#\bfP$, defined by
\[
U_m = U_\bfP = \left\{\Phi^{[d]}_j \mid 1 \le j \le d, \quad 1 \le i \le \#\bfP\right\}
\]
for a given partition $\bfP$ of $\Omega$. We have~$d$ actuators~$\Phi^{[d]}_j$, $1\le j\le d$, supported on a common subdomain $\omega_i$ congruent to~$\omega$, $\omega_i\subseteq\bfP^i\subseteq\Omega$ for each element~$\bfP^i\in\bfP$ as illustrated in Fig.~\ref{fig:act}.

To achieve stabilizability we will  require a sufficiently large number $M$ of actuators. However, by construction the total volume occupied by all actuators is fixed a priori. This is possible because we decrease proportionally the volume covered by each actuator as we increase~$M$; see Fig.~\ref{fig:act}.

\begin{remark} In this manuscript we consider smooth domains in order to cover several types of boundary conditions. Some results in the literature, depending on these conditions, are stated for regular boundaries only (e.g., the result in~\cite[Appen.~I, Prop.~1.4]{Temam01} used to deal with curl-type boundary conditions). Our results will be valid  for polygonal/polyhedral domains as well, if the existence  of weak solutions holds for such domains with the given boundary conditions. Note that we can partition a polygonal domain into a finite number of fixed triangles and/or squares, subsequently we can partition each of those subrectangles/subtriangles as in Fig.~\ref{fig:act-recttri} (thus, obtaining partitions with a finite number of subdomains only, up to a translation, a rotation, and a homothety; see~\cite[Rem.~2.8]{AzmiKunRod23-tac}).
\end{remark}
\begin{figure}[htbp]%
    \centering%
\subfigure[Placement in a rectangle.]
    {\includegraphics[width=.8\textwidth]{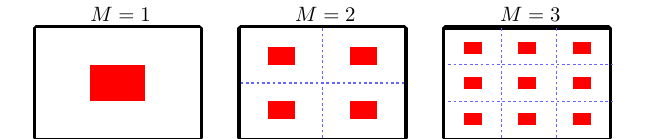}}\\
   \subfigure[Placement in a triangle.] 
    {\includegraphics[width=.85\textwidth]{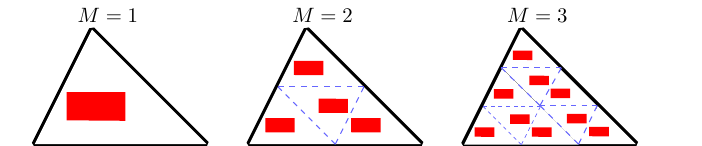}} \\
  \caption{Placement  of the supports of the actuators in rectangles (a) and triangles (b).}%
     \label{fig:act-recttri}%
\end{figure}

\subsubsection{The explicit feedback-input operator}
In this manuscript, we employ a feedback-input control law as follows. Let $\Pi$ denote the Helmholtz projection onto the subspace~$H\subset L^2(\Omega,\bbR^d)$,
\begin{equation}\label{spaceH}
    H\coloneqq\{h\in L^2(\Omega,\bbR^d)\mid\; \diver h=0\mbox{ and }( y  \cdot \bfn)\rest{\partial\Omega} = 0\}.
\end{equation}
of divergence-free vector fields tangent to the boundary~$\partial\Omega$ of the spatial domain~$\Omega$.
 The feedback-input operator is defined by
\begin{equation}\label{Feed-orth}
\Pi U_m^\diamond K_m^\lambda \coloneqq -\lambda P_{\Pi\mathcal{U}_m},
\end{equation}
where~$\lambda\ge0$ is a scalar and $P_{\Pi\mathcal{U}_m} \colon H \to \Pi\mathcal{U}_m$ is the orthogonal projection in $H$ onto $\Pi\mathcal{U}_m$, $\mathcal{U}_m = \text{span}\,U_m$, and the product~$\Pi U_m^\diamond$
is the control operator, with $U_m^\diamond \colon u \mapsto \sum_{i=1}^{m}u_i\Phi_i$.

Equivalently, the feedback law can be written explicitly as
\begin{equation}\label{Feed-intro}
K_m^\lambda y \coloneqq -\lambda [\mathcal{V}_m]^{-1}\begin{bmatrix}
    (\Pi\Phi_1,y)_{H}\\
    \vdots\\
    (\Pi\Phi_m,y)_{H}
\end{bmatrix},
\end{equation}
where the matrix $\mathcal{V}_m$ has entries
\[
[\mathcal{V}_m]_{ij} = (\Pi\Phi_i,\Pi\Phi_j)_{H}, \quad 1 \le i,j \le m,
\]
forming a diagonal matrix with diagonal entries $[\mathcal{V}_m]_{ii}=\norm{\Pi\Phi_i}{H}^2$ for $1 \le i \le m$.

\subsubsection{Stabilizability}
We are now ready to state our main result, for a disk as above in Fig.~\ref{fig:act} or polygons as in Fig.~\ref{fig:act-recttri}. System~\eqref{nse-y-intro} under the action of the feedback-input~\eqref{Feed-intro} reads
\begin{subequations}\label{nse-yK-intro}
\begin{align}
&\tfrac{\partial}{\partial t} \bfy =\nu\Delta\bfy -\langle\bfy \cdot\nabla\rangle \bfy  + \nabla p =\bff+U_M^\diamond K_M^\lambda(\bfy-\widehat\bfy)\qquad \diver \bfy = 0, \label{nse-yK-intro-dyn}\\
&\beta\tfrac{\partial}{\partial t}\bfy=-  \fkD\bfy-\nu\fkL\bfy+ \bfg,\qquad (\bfy  \cdot \bfn)\rest{\partial\Omega} = 0,\label{nse-yK-intro-bdry}\\
&\bfy(\cdot,0) =\bfy_0,\qquad\bfy(\cdot,0)\rest{\partial\Omega}=\bfy_{0\partial}. \label{nse-yK-intro-ic}
\end{align}
\end{subequations}

\begin{theorem}\label{th:main-intro}
Given a target trajectory $\widehat{\bfy} \in L^\infty(\mathbb{R}_+,V)$, there exists~$M_*\in\bbN$ such that for every~$M \ge M_*$, there exists~$\lambda_* = \lambda_*(M) \ge 0$ such that, for every~$\lambda\ge\lambda_*$ the solution to~\eqref{nse-yK-intro} satisfies the exponential stabilization~\eqref{goal-intro-exp} with decay rate $\mu = \beta^{-1}\alpha$, where $\alpha$ is defined in~\eqref{fkDalpha}, $X = L^2(\Omega,\bbR^d) \times L^2(\partial\Omega,\mathbf{T}\partial\Omega)$ and $C = 1$.
\end{theorem}

The proof will be presented in Section~\ref{sec:stabil}. A more general version of this result, stated as Theorem~\ref{th:main}, applies to arbitrary domains $\Omega$ and actuator supports $\omega$ satisfying some regularity conditions as formalized in Assumption~\ref{ass:domain}.

By Theorem~\ref{th:main-intro} we have that the stabilization decay rate $\mu = \beta^{-1}\alpha$ is achievable. Note that, as $\beta$ tends to~$0$, we see that the rate~$\mu$ diverges to infinity, suggesting that arbitrarily fast stabilization can be obtained in the limit of static boundary conditions. This behavior is consistent with existing results in the literature stating that under static ($\beta = 0$) boundary conditions, arbitrarily large stabilization rates are achievable, as demonstrated for curl-type static boundary conditions in~\cite{RodSeifu24}.

In Theorem~\ref{th:main-intro}, the fact that~\eqref{goal-intro-exp} holds with~$C=1$ is interesting, because it implies that the norm~$\norm{\bfY(\cdot,t) -\widehat\bfY(\cdot,t) }{X}$ is strictly decreasing.

	\section{Mathematical setting}\label{sec:mathSett}
 We recall the operator~$\overline\fkL$ in~\eqref{barL2} and the space~$H$ in~\eqref{spaceH} and introduce the spaces
  \begin{subequations}\label{spacesHV}
  \begin{align}
 &\bfH= H\times L^2(\partial\Omega,\bfT\partial\Omega), &&\\
 &V\coloneqq\{h\in H\mid\; \overline\fkL h\in L^2(\Omega,\bbR^{d_1})\},\quad\bfV\coloneqq \{(h,\varphi)\in V\times L^2(\partial\Omega,\bfT\partial\Omega)\mid h\rest{\partial\Omega}=\varphi\}. 
 \end{align}
 \end{subequations}
 Following a standard procedure, we start by projecting the operators defining the dynamics onto~$H$. In this way we can write~\eqref{nse-haty-intro} and~\eqref{nse-yK-intro} as evolutionary equations in~$H$. Further,  by considering $y\coloneqq\bfy-\widehat\bfy$ the stabilization of~$\bfy$ to~$\widehat\bfy$ is reduced to the stabilization of~$y$ to zero. The evolutionary equation for~$(y,\varphi)\in \bfV$ reads
  \begin{subequations}\label{nse-dif}
\begin{align}
&\dot y =\nu\Pi\Delta y -B(\widehat\bfy,y)-B(y,\widehat\bfy)-B(y,y)  +\Pi U_m^\diamond K_m^\lambda y,\qquad \diver y = 0, \label{nse-dif-dyn}\\
& \beta\dot \varphi =  -\fkD \varphi-\nu\fkL y,\qquad ( y  \cdot \bfn)\rest{\partial\Omega} = 0,\label{nse-dif-bdry}\\
&(y(0),\varphi(0)) =(y_0,\varphi_0)\in \bfH. \label{nse-dif-ic}
\end{align}
\end{subequations} 
and our goal~\eqref{goal-intro-exp} becomes, when~$X=\bfH$,
\begin{equation}\label{goal-dif-exp}
\norm{(y,\varphi)(t) }{\bfH}\le C\rme^{-\mu(t-s)}\norm{(y,\varphi)(s)}{\bfH},\quad\mbox{for all }t\ge s\ge 0.
\end{equation}

Hereafter, we require the following. 
\begin{assumption}\label{ass:Green}
    The formula~\eqref{barL1} holds for smooth vector fields in~$V$.
\end{assumption}
\begin{assumption}\label{ass:VnormW12}
    We have that~$z\mapsto(\norm{\overline\fkL z}{L^2}^2+\norm{z}{L^2}^2)^{\frac12}$ defines a norm on~$V$ equivalent to the Sobolev $W^{1,2}(\Omega,\bbR^d)$-norm.
\end{assumption}
\begin{assumption}\label{ass:VcompH}
    We have that~$\bfV\xhookrightarrow{\rm d,c}\bfH$.
\end{assumption}

\subsection{The extended shifted Stokes operator}
Often we shall denote elements~$(w,\phi)\in\bfH$ in matrix form. With such notation, we define the bilinear form
\begin{subequations}\label{form4Stokes0}
\begin{align}
  &\fka\colon \bfV\times \bfV\to\bbR, \\
  &\fka\left(\begin{bmatrix}
      w\\\phi
  \end{bmatrix},\begin{bmatrix}
      z\\\psi
  \end{bmatrix}\right)\coloneqq \nu(\overline\fkL w,\overline\fkL z)_{L^2(\Omega,\bbR^{d_1})} +(w,z)_{L^2(\Omega,\bbR^{d})}+(\fkD\phi,\psi)_{L^2(\partial\Omega,\bfT\partial\Omega)},
\end{align}
which, by Assumption~\ref{ass:VcompH} and~\eqref{fkDalpha}, defines a scalar product on~$\bfV$.
\end{subequations}

Now, we consider the extended Stokes operator~$A$ defined by
\begin{subequations}\label{operA0Prod}
 \begin{equation}\label{operA0}
\left\langle A \begin{bmatrix}
      w\\\phi
  \end{bmatrix},\begin{bmatrix}
      z\\\psi
  \end{bmatrix}\right\rangle_{\bfV',\bfV}\coloneqq \fka\left(\begin{bmatrix}
      w\\\phi
  \end{bmatrix},\begin{bmatrix}
      z\\\psi
  \end{bmatrix}\right).
\end{equation}

From now on, we consider~$\bfV$ endowed with the scalar product
\begin{equation}\label{scprodA0}
\left(\begin{bmatrix}
      w\\\phi  \end{bmatrix},\begin{bmatrix}
      z\\\psi  \end{bmatrix}\right)_\bfV\coloneqq\fka\left(\begin{bmatrix}
      w\\\phi  \end{bmatrix},\begin{bmatrix}
      z\\\psi  \end{bmatrix}\right)
\end{equation}
\end{subequations}

\begin{lemma}\label{lem:isomVH}
    The operator~$A\colon\bfV\mapsto\bfV'$ is a bijective isometry.
\end{lemma}
\begin{proof}
The operator~$A$ is bijective because it defines a scalar product on~$\bfV$ and due to the Lax-Milgram lemma. That $A$ is an isometry follows straighforwardly from~\eqref{operA0Prod}.
\end{proof}

As usual the domain of~$A$ is defined as
\[
\rmD(A)\coloneqq\left\{\left.\begin{bmatrix}
      w\\\phi
  \end{bmatrix}\in\bfH\;\right|\; A\begin{bmatrix}
      w\\\phi
  \end{bmatrix}\in\bfH\right\}.
\]
\begin{lemma}
The domain of~$A$ is given by
\[\rmD(A)=D_A\coloneqq\left\{\left.\begin{bmatrix}
      w\\\phi
  \end{bmatrix}\in\bfV\; \right|\;\Pi\Delta w\in H\mbox{ and }\fkD\phi +\nu\fkL w \in L^2(\partial\Omega,\bfT\partial\Omega)\right\}.\]
\end{lemma}
\begin{proof}

Let~$(z,\psi)\in D_A$ be arbitrary. By Assumption~\ref{ass:Green} we have~\eqref{barL1}, which combined with~\eqref{form4Stokes0} and~\eqref{operA0Prod} gives us that~$(z,\psi)\in\rmD(A)$ and, consequently,~$D_A\subseteq\rmD(A)$. Next, let~$(w,\phi)\in\rmD(A)$ be arbitrary. Given~$z\in V^\infty_\rmc\coloneqq\{h\in V\mid h\text{ smooth with support}\supp(h)\subset\Omega\}$, we find that~$\left\langle A \begin{bmatrix}
      w\\\phi
  \end{bmatrix},\begin{bmatrix}
      z\\0
  \end{bmatrix}\right\rangle=-\nu\langle \Delta w,z\rangle$. Then, since~$A(w,\phi)\in\bfH$ and~$V^\infty_\rmc$ is dense in~$H$ we must have that~$\Pi\Delta w\in H'=H$. Further, using~\eqref{barL1} we  have~$\left\langle A \begin{bmatrix}
      w\\\phi
  \end{bmatrix},\begin{bmatrix}
      z\\\psi
  \end{bmatrix}\right\rangle=-\nu\langle \Delta w,z\rangle+\nu\langle\fkL w,\psi\rangle+\langle w,z\rangle+\langle\fkD \phi,\psi\rangle $. Again, since~$A(w,\phi)\in\bfH$ we have that~$\fkD \phi+\nu\fkL w\in L^2(\partial\Omega,\bfT\partial\Omega)$. Therefore, we have~$\rmD(A)\subseteq D_A$.
\end{proof}

By Assumption~\ref{ass:VcompH} and Lemma~\ref{lem:isomVH}, it also follows that~$A\colon\rmD(A)\to \bfH$ is an isometry and that~$\rmD(A)\xhookrightarrow{\rm c,d}\bfH$. Consequently, $A^{-1}\colon \bfH\to\bfH$ is compact and we conclude that we have a sequence of eigenpairs~$((\gamma_j,e_j))_{j\in\bbN_+}$,   
\begin{equation}\label{eigenpairs}
Ae_j=\gamma_je_j,\qquad 0<\gamma_1\le....\le\gamma_j \le \gamma_{j+1},\quad\gamma_j\to\infty,
\end{equation}
where~$e_j$ is an orthonormal complete basis in~$\bfH$.

\subsection{The reaction-convection components}
We define the linear operator
\begin{equation}\label{operArc}
         A_{\rm rc}\colon \bfV\to\bfV',\qquad
   A_{\rm rc}\begin{bmatrix}
      w\\\phi  \end{bmatrix}\coloneqq\begin{bmatrix}B(\widehat\bfy,w) + B(w,\widehat\bfy) \\ 0 \end{bmatrix}-\begin{bmatrix} w\\0\end{bmatrix}
\end{equation}
and the nonlinearity
\begin{equation}\label{operN}
         \clN\colon \bfV\to\bfV',\qquad
   \clN\left(\begin{bmatrix}
      w\\\phi
  \end{bmatrix}\right)\coloneqq\begin{bmatrix} B(w,w) \\ 0 \end{bmatrix}.
\end{equation}

\subsection{The feedback control forcing}
We write the control forcing as
\begin{equation}\label{operFeed}
         F_m^\lambda\colon \bfV\to\bfV',\qquad
   F_m^\lambda\begin{bmatrix}
      w\\\phi  \end{bmatrix}\coloneqq\begin{bmatrix}\Pi U_m^\diamond K_m^\lambda w\\0\end{bmatrix},
\end{equation}
with~$K_m^{\lambda}$ as in~\eqref{Feed-intro}.

\subsection{Evolutionary extended system}
Finally, recalling~\eqref{barL1}, we write~\eqref{nse-dif} as  \begin{subequations}\label{nse-dif-evol}
\begin{align}
&\tfrac{\rmd}{\rmd t}\begin{bmatrix}
      y\\\beta\varphi
  \end{bmatrix}=-A Y-A_{\rm rc} Y-\clN\left( Y\right)+F_m^\lambda Y,\qquad Y\coloneqq\begin{bmatrix}
      y\\\varphi
  \end{bmatrix},\\
&  Y(0)=Y_0,
\end{align}
\end{subequations}
and we recall also that our goal is to show that the feedback forcing~$F_m^\lambda$ stabilizes the system, as in~\eqref{goal-dif-exp}, for large enough~$m$ and~$\lambda$.

\begin{remark}
    The term~$(w,z)_H$ added in~\eqref{form4Stokes0} is useful to guarantee the positive-definiteness of~$\fka$. It is equivalent to add~$(w,\varphi)\mapsto(w,0)$ to the Stokes operator. Note than this term is then subtracted in the reaction-convection term in~\eqref{operArc}.
\end{remark}

\section{Existence of solutions}\label{sec:well-posed}
The existence of variational solutions can be shown by a limit of Galerkin approximations as in \cite[Ch.~3, Sect.~3.2]{Temam01}.
We consider a more general class of systems in order to include both the free~\eqref{nse-haty-intro} and controlled~\eqref{nse-dif-evol} dynamics, for this purpose, we consider~\eqref{nse-dif-evol} with an external forcing~$(f,g)\in L^2_{\rm loc}(\bbR_+,\bfH)$ as follows.
\begin{subequations}\label{nse-gen-evol}
\begin{align}
&\tfrac{\rmd}{\rmd t}\begin{bmatrix}
      y\\\beta\varphi
  \end{bmatrix}=-A Y-A_{\rm rc} Y-\clN\left( Y\right)+F_m^\lambda Y+\begin{bmatrix}
     f \\ g
  \end{bmatrix},\qquad Y=\begin{bmatrix}
      y\\\varphi
  \end{bmatrix},\\
&  Y(0) = Y_0\in \bfH.
\end{align}
\end{subequations}

We make the following regularity assumption.
\begin{assumption}\label{ass:target}
 The target trajectory $\widehat{\bfy}$, solving~\eqref{nse-haty-intro}, satisfies
 \begin{equation*}
  \norm{\nabla\widehat\bfy}{ L^\infty(\bbR_+,L^2(\Omega,\bbR^{d^2}))}\eqqcolon\widehat C_2<\infty.   
 \end{equation*}
  \end{assumption} 

Following a standard strategy, we seek  an approximate solution~$Y^N(t)=(y^N,\varphi^N)(t)\in \clE_N$ evolving  in the finite-dimensional space~$\clE_N\coloneqq\linspan\{e_j\mid 1\le j\le N\}$ spanned by the first eigenfunctions~$e_j$ as in~\eqref{eigenpairs}. Namely, solving
\begin{subequations}\label{nse-GalN}
\begin{align}
&\tfrac{\rmd}{\rmd t}\begin{bmatrix}
      y^N \\ \beta\varphi^N
  \end{bmatrix} = P_N\left(-A Y^N - A_{\rm rc} Y^N- \clN \left( Y^N \right) + F_m^\lambda Y^N+\begin{bmatrix}
      f \\ g
  \end{bmatrix}\right),\\
&  Y^N(0)=P_NY^N_0,\qquad Y^N=\begin{bmatrix}
      y^N\\\varphi^N
  \end{bmatrix},
\end{align}
\end{subequations}
where~$P_N\colon \bfH\to \clE_N$ is the orthogonal projection in~$\bfH$ onto~$\clE_N$. Existence of solutions for this finite-dimensional system follows from Carath\'eodory  theorem~\cite[Ch.~I, Thm.~5.3]{Hale80}.

Now, for~$(z,\phi)\in\bfH$, we define
\begin{equation}\label{normbetaH}
    \norm{(z,\phi)}{\bfH,\beta}^2\coloneqq \norm{z}{H}^2 + \beta\norm{\phi}{L^2(\partial\Omega,\bfT\partial\Omega)}^2.
\end{equation}
By testing the dynamics with~$2Y^N$, we obtain 
\begin{align*}
\tfrac{\rmd}{\rmd t} \norm{Y^N}{\bfH,\beta}^2&=- 2\norm{Y^N}{\bfV}^2-2\langle A_{\rm rc} Y^N+F_m^\lambda Y^N+f,y^N\rangle_{\bfV',\bfV}+2( g,\varphi^N)_{L^2(\partial\Omega,\bfT\partial\Omega)}
\\
&\le  - 2\norm{Y^N}{\bfV}^2+2C_{\rm rc}\norm{y^N}{H}^{\frac12}\norm{y^N}{V}^{\frac32}+2\norm{y^N}{H}^{2}+2\norm{F_m^\lambda}{\clL(H)}\norm{y^N}{H}^2\\
&\quad+2\norm{f}{V'}\norm{y^N}{V}+2\norm{g}{L^2(\partial\Omega,\bfT\partial\Omega)}\norm{\varphi^N}{L^2(\partial\Omega,\bfT\partial\Omega)}
\end{align*}
and, by the Young inequality,
\begin{align*}
\tfrac{\rmd}{\rmd t} \norm{Y^N}{\bfH,\beta}^2
&\le  - \norm{Y^N}{\bfV}^2 +\left(\tfrac14\Bigl((\tfrac{2}{3})^{-\frac34}C_{\rm rc}\Bigr)^4+2+2\norm{F_m^\lambda}{\clL(H)}\right)\norm{y^N}{H}^2\\
&\quad+\norm{\varphi^N}{L^2(\partial\Omega,\bfT\partial\Omega)}^2+2\norm{f}{V'}^2+\norm{g}{L^2(\partial\Omega,\bfT\partial\Omega)}^2
\\
&\le  - \norm{Y^N}{\bfV}^2 +C_1\norm{Y^N}{\bfH,\beta}^2+2\norm{f}{V'}^2+\norm{g}{L^2(\partial\Omega,\bfT\partial\Omega)}^2
\end{align*}
with~$C_1\coloneqq \left(\tfrac14((\tfrac{2}{3})^{-\frac34}C_{\rm rc})^4+2+2\norm{F_m^\lambda}{\clL(H)}+\beta^{-1}\right)$. The Gronwall lemma and time integration give us
\begin{align}\label{bddGalN-infH2V}
&\norm{Y^N}{L^\infty(0,T;\bfH)}^2+\norm{Y^N}{L^2(0,T;\bfV)}^2\le C_2,
\end{align}
with~$C_2$ independent of~$N$ and depending on
\[
\norm{(Y_0,f,g)}{\bfH\times L^2(0,T;\bfV')\times L^2(0,T;L^2(\partial\Omega,\bfT\partial\Omega)}.
\]
By the dynamics in~\eqref{nse-GalN} and~$\bfV\xhookrightarrow{}\bfH$, we find
\begin{align*}
\norm{\dot Y^N}{\bfV'}&\le \norm{Y^N}{\bfV}+\norm{A_{\rm rc} Y^N}{\bfV'}+\norm{\clN\left( Y^N\right)}{\bfV'}+\norm{F_m^\lambda Y^N}{\bfV'}+\norm{(f,g)}{V'\times L^2(\partial\Omega,\bfT\partial\Omega)}\\
&\le C_3\norm{Y^N}{\bfV}+\norm{\clN\left( Y^N\right)}{\bfV'}+\norm{(f,g)}{V'\times L^2(\partial\Omega,\bfT\partial\Omega)}\\
&\le C_3\norm{Y^N}{\bfV}+C_4\norm{Y^N}{\bfV}^\frac32\norm{Y^N}{\bfH}^\frac12+\norm{(f,g)}{V'\times L^2(\partial\Omega,\bfT\partial\Omega)}
\end{align*}
which gives us, using~\eqref{bddGalN-infH2V},
\begin{align*}
\norm{\dot Y^N}{L^{\frac43}(0,T;\bfV')}&\le C_5\norm{(Y^N,f,g)}{L^{2}(0,T;\bfV\times V'\times L^2(\partial\Omega,\bfT\partial\Omega))}\le C_6
\end{align*}
with~$C_6$ independent of~$N$.

By the Aubin--Lions--Simon lemma in~\cite[Sect.~9, Cor.~6]{Simon86} we have that
\begin{align}
\clW\coloneqq\{h\in L^2(0,T;\bfV)\mid h\in L^\infty(0,T;\bfH)\mbox{ and }\dot h\in L^{\frac43}(0,T;\bfV')\}
\xhookrightarrow{\rm c} L^2(0,T;\bfH).
\end{align}

Now, we can take a subsequence (still denoted~$(Y^N)_{N\in\bbN_+}$) so that
\begin{align}
&Y_N\xrightharpoonup[L^2(0,T;\bfV)]{} Y_\infty,\quad &&Y_N\xrightharpoonup[L^\infty(0,T;\bfH)\;]{}\hspace{-.5em}{_*}\; Y_\infty,\\
&\dot Y_N\xrightharpoonup[L^{\frac43}(0,T;\bfV')]{} Y_\infty,\quad &&Y_N \xrightarrow[L^2(0,T;\bfH)]{}Y_\infty,
\end{align}
for some~$Y^\infty\in\clW$. From this (using~\cite[Lem.~3.2]{Temam01} for the nonlinear term) we can pass to the limit of the dynamics in 
$L^{\frac43}(0,T;\bfV')$ to obtain
\begin{align*}
&\tfrac{\rmd}{\rmd t}\begin{bmatrix}
      y^\infty \\ \beta \varphi^\infty
  \end{bmatrix} = -A Y^\infty-A_{\rm rc} Y^\infty-\clN( Y^\infty) + F_m^\lambda Y^\infty + \begin{bmatrix}
      f \\ g
  \end{bmatrix},\qquad Y^\infty=\begin{bmatrix}
      y^\infty\\\varphi^\infty
  \end{bmatrix}.
\end{align*}
Finally, 
the relation~$Y^\infty(0)=Y_0$ is justified by following the arguments in~\cite[Ch.~3, Sect.~3.2, Eqs.~(3.43)--(3.44)]{Temam01}. That is, $Y^\infty\in\clW$ solves~\eqref{nse-gen-evol}.

\section{Monotonicity of the sum Stokes plus feedback}\label{sec:monotStokes}
We derive auxiliary properties for the operator
\begin{subequations}\label{fkA}
\begin{equation}
\fkA_m^\lambda\coloneqq A_\rmS-2\Pi U_m^\diamond K_m^\lambda=A_\rmS+2\lambda P_{\Pi\clU_m},
\end{equation}
with~$K_m^\lambda$ as in~\eqref{Feed-intro} and
\begin{equation}
\langle A_\rmS v,z\rangle_{V',V}\coloneqq \nu(\overline\fkL v,\overline\fkL z)_{L^2(\Omega,\bbR^{d_1})}+(v,z)_{L^2(\Omega,\bbR^{d})}.
\end{equation}
\end{subequations}

\subsection{Scalar actuators}

We need the following assumption on~$\Omega$ and the reference support~$\omega_*$ of the actuators.
\begin{assumption}\label{ass:domain}
    The open domains~$\Omega$ and~$\omega_*$ satisfy the following: for every~$h>0$, there exists a partition~$\fkP_h=\{\clO_i\mid 1\le i\le m\}$ of~$\Omega\subset\bbR^d$,
    $\overline\Omega=\cup_{i=1}^m\overline\clO_i$,
    with pairwise disjoint open domains $\clO_i$, such that up to 
    a translation and a rotation of~$\clO_i$.
     \[h\omega_*\subset \clO_i\subset h\fkB_1,
     \]
    and there exist extension operators $E_i$ such that
    \begin{equation}
    E_i\colon W^{1,2}(\clO_i)\to W^{1,2}(\bbR^d) \quad  \text{ with } \quad \norm{E_i}{\clL(W^{1,2}(\clO_i)\to W^{1,2}(\bbR^d))}\le C_E    
    \end{equation}
     with uniform bound, that is, with~$C_E$ independent of~$i$ and~$h$.
\end{assumption}

Now we state the main result of this section.
\begin{theorem}\label{thm:Eig1-AFeed}
Let Assumption~\ref{ass:domain} hold true and let~$\zeta>0$. Then,  there exists~$h_*>0$ such that for every $h\le h_*$ and every partition~$\fkP_h=\{\clO_i\mid 1\le i\le m\}$ of~$\Omega\subset\bbR^d$ as in Assumption~\ref{ass:domain},
there exists~$\lambda_*>0$ such that for every~$\lambda\ge\lambda_*$, we have
\[
\Lambda_1(\fkA_m^\lambda)\coloneqq\min_{v\in V}\frac{\langle \fkA_m^\lambda v,v\rangle_{V',V}}{\norm{v}{H}}\ge\zeta.
\]
\end{theorem}
The proof is given below in Section~\ref{sS:proofthm:Eig1-AFeed}. First, we prove several auxiliary results that are essential to the argument.

\begin{remark}
    Theorem~\ref{thm:Eig1-AFeed} tells us that the first eigenvalue~$\Lambda_1(\fkA_m^\lambda)$ of the symmetric operator~$\fkA_m^\lambda$ can be made arbitrarily large by choosing~$h$ small enough and~$\lambda$ large enough. The number~$m=m(h)$ of elements in the partition satisfies $m\to\infty$ as~$h\to0$.
\end{remark}

\subsection{Auxiliary Poincar\'e-like inequalities}
For a partition as in Assumption~\ref{ass:domain},
\begin{subequations}\label{omega_iPh}
 \begin{equation}
  \fkP_h=\{\clO_i\mid 1\le i\le m\},\quad\text{satisfying}\quad \omega_i\subset \clO_i\subset h\fkB_1, 
\end{equation}
where~$\omega_i=h\omega_*$, up to a translation and a rotation, we define the set
\begin{equation}
U_m^0\coloneqq\{\indf_{\omega_i}\mid 1\le i\le m\},\qquad\clU_m^0\coloneqq\linspan U_m^0,
\end{equation}
\end{subequations}
 of indicator functions~$\indf_{\omega_i}$ and the orthogonal projection~$P_{\clU_m^0}$ in~$L^2(\Omega)$ onto~$\clU_m^0$.

\begin{lemma}\label{lem:poincare0}
 Let Assumption~\ref{ass:domain} hold true. Then, for~$m$ indicator functions as in~\eqref{omega_iPh}, we have that the Poincar\'e-like constant
 \[
\beta_{\fkP_h}^0\coloneqq\inf_{f\in (W^{1,2}(\Omega)\cap \clU_{m}^{0,\perp})\setminus\{0\}}\frac{\norm{f}{W^{1,2}(\Omega)}^2}{\norm{f}{L^2(\Omega)}^2}
\]
satisfies~$\lim_{h\to 0}\beta_{\fkP_h}^0=\infty$.
\end{lemma}

\begin{proof}
With~$h\omega_*\subset \clO_i\subset h\fkB_1$, let
\[
C_h\coloneqq \inf_{g\in W^{1,2}(h\fkB_1)\cap\{\indf_{h\omega_*}\}^\perp\setminus\{0\}}\frac{\norm{g}{W^{1,2}(h\fkB_1)}^2}{\norm{g}{L^2(h\fkB_1)}^2}.
\]
Then we find, with~$X_0\coloneqq W^{1,2}(\clO_i)\cap\{\indf_{\omega_i}\}^\perp\setminus\{0\}$,
\[
D\coloneqq \inf_{f\in X_0}\frac{\norm{f}{W^{1,2}(\clO_i)}^2}{\norm{f}{L^2(\clO_i)}^2}\ge\inf_{f\in X_0}\frac{\norm{f}{W^{1,2}(\clO_i)}^2}{\norm{E_if}{L^2(h\fkB_1)}^2}\ge C_E^{-1}\inf_{f\in X_0}\frac{\norm{E_if}{W^{1,2}(h\fkB_1)}^2}{\norm{ E_if}{L^2(h\fkB_1)}^2}\ge C_E^{-1}C_h.
\]
Next, with~$X_1\coloneqq (W^{1,2}(\Omega)\cap\clU_{m}^{0,\perp})\setminus\{0\}$,
 \[
\beta_{\fkP_h}^0=\inf_{f\in X_1}\frac{\norm{f}{W^{1,2}(\Omega)}^2}{\norm{f}{L^2(\Omega)}^2}=\inf_{f\in X_1}\frac{\sum_{i=1}^m\norm{f}{W^{1,2}(\clO_i)}^2}{\norm{f}{L^2(\Omega)}^2}\ge \inf_{f\in X_1}\frac{C_E^{-1}C_h\sum_{i=1}^m\norm{f}{L^2(\clO_i)}^2}{\norm{f}{L^2(\Omega)}^2}=C_E^{-1}C_h.
\]
Finally, recall that~$C_h\to\infty$ as~$h\to0$ (e.g., see the arguments in~\cite[Sect.~5]{Rod21-sicon}).
\end{proof}

\begin{corollary}\label{cor:poincare-L2}
With the~$dm=\#U_m$ actuators constructed as in~\eqref{Act-intro}, with~$m$ subdomains~$\omega=\omega_i$ as in~\eqref{omega_iPh} we have that
\[
\beta_{\fkP_h}\coloneqq\inf_{v\in (W^{1,2}(\Omega,\bbR^d)\cap \clU_{m}^\perp)\setminus\{0\}}\frac{\norm{v}{W^{1,2}(\Omega,\bbR^d)}^2}{\norm{v}{L^2(\Omega,\bbR^d)}^2}\longrightarrow\infty.
\]    
\end{corollary}

\begin{proof}
    Writing~$(v_1,...,v_d)\coloneqq v$, from Lemma~\ref{lem:poincare0}
    we have that the scalar coordinates satisfy
    \[
    \beta_{\fkP_h}^0=\inf_{v_i\in (W^{1,2}(\Omega)\cap \clU_{m}^{0,\perp})\setminus\{0\}}\frac{\norm{v_i}{W^{1,2}(\Omega)}^2}{\norm{v_i}{L^2(\Omega)}^2}\longrightarrow\infty.
    \]
    Then, with~$Z\coloneqq(W^{1,2}(\Omega,\bbR^d)\cap \clU_{m}^\perp)\setminus\{0\}$, we find
  \begin{align}
 \beta_{\fkP_h}=\inf_{h\in Z}\frac{\norm{h}{W^{1,2}(\Omega,\bbR^d)}^2}{\norm{h}{L^2(\Omega,\bbR^d)}^2}=\inf_{h\in Z}\frac{\sum_{i=1}^d\norm{h_i}{W^{1,2}(\Omega)}^2}{\sum_{i=1}^d\norm{h_i}{L^2(\Omega)}^2}
\ge\inf_{h\in Z}\frac{\sum_{i=1}^d\norm{h_i}{W^{1,2}(\Omega)}^2}{d\norm{h_I}{L^2(\Omega)}^2}  \notag
  \end{align}
  with~$I\in\argmax\norm{h_i}{L^2(\Omega)}$. Hence, it follows
   \begin{align}
 \beta_{\fkP_h}
\ge\inf_{h\in Z}\frac{\norm{h_I}{W^{1,2}(\Omega)}^2}{d\norm{h_I}{L^2(\Omega)}^2}\ge\frac1d\beta_{\fkP_h}^0\xrightarrow[\quad]{}\infty, \notag
  \end{align} 
  which finishes the proof.
  \end{proof}
\begin{corollary}\label{cor:poincare-H}
Let~$H_1\coloneqq H\cap W^{1,2}(\Omega,\bbR^d)$, then
\[
\beta_{\fkP_h}^H\coloneqq\inf_{v\in (H_1\cap (\Pi \clU_m)^\perp)\setminus\{0\}}\frac{\norm{v}{W^{1,2}(\Omega,\bbR^d)}^2}{\norm{v}{L^2(\Omega,\bbR^d)}^2}\longrightarrow\infty.
\]
\end{corollary}
\begin{proof}
    Observe that for arbitrary $h\in (H_1\cap (\Pi U_m)^\perp)\setminus\{0\}$ and~$g\in U_m$ we have
   \[
   (h,g)_{L^2(\Omega,\bbR^d)}=(h,\Pi g)_{H}=0. 
   \]
   Thus, in particular, it follows that~$
   h\in (W^{1,2}(\Omega,\bbR^d)\cap  \clU_m^\perp)\setminus\{0\}$.
   Since
   $H_1\subset W^{1,2}(\Omega,\bbR^d)$, it follows that
   \[
\beta_{\fkP_h}^H\coloneqq\inf_{h\in (H_1\cap (\Pi \clU_m)^\perp)\setminus\{0\}}\frac{\norm{h}{W^{1,2}(\Omega,\bbR^d)}^2}{\norm{h}{L^2(\Omega,\bbR^d)}^2}\ge \inf_{h\in (W^{1,2}(\Omega,\bbR^d)\cap \clU_{m}^\perp)\setminus\{0\}}\frac{\norm{h}{W^{1,2}(\Omega,\bbR^d)}^2}{\norm{h}{L^2(\Omega,\bbR^d)}^2}=\beta_{\fkP_h}.
\]
By Corollary~\ref{cor:poincare-L2} it follows that~$\beta_{\fkP_h}^H\ge\beta_{\fkP_h}\longrightarrow\infty$.
\end{proof}
\begin{corollary}\label{cor:poincare-AVH}
We have that
\[
\beta_{\fkP_h}^V\coloneqq\inf_{h\in (V\cap (\Pi \clU_m)^\perp)\setminus\{0\}}\frac{\norm{h}{V}^2}{\norm{h}{H}^2}\longrightarrow\infty.
\]
\end{corollary}
\begin{proof}
 Straightforward from  Corollary~\ref{cor:poincare-H}, from~$V\subseteq H_1=H\cap W^{1,2}(\Omega,\bbR^d)$, and from the fact that, on~$V$, the~$V$-norm is equivalent to the~$W^{1,2}(\Omega,\bbR^d)$-norm; recall the discussion following~\eqref{form4Stokes0}. Indeed we obtain, with some~$C>0$,
 \begin{align*}
  \beta_{\fkP_h}^V &\ge C \inf_{h\in (V\cap (\Pi \clU_m)^\perp)\setminus\{0\}}\frac{\norm{h}{W^{1,2}(\Omega,\bbR^d)}^2}{\norm{h}{H}^2}\ge C \inf_{h\in (H_1\cap (\Pi \clU_m)^\perp)\setminus\{0\}}\frac{\norm{h}{W^{1,2}(\Omega,\bbR^d)}^2}{\norm{h}{H}^2}=C\beta_{\fkP_h}^H.
 \end{align*}
By Corollary~\ref{cor:poincare-H}, we find~$\beta_{\fkP_h}^V\ge C\beta_{\fkP_h}^H\longrightarrow\infty$.
 \end{proof}

\subsection{Proof of Theorem~\ref{thm:Eig1-AFeed}}\label{sS:proofthm:Eig1-AFeed}
Let~$\zeta>0$ be given. We write
\begin{align}\label{estEig1-0}
\Lambda_1(\fkA_m^\lambda)=\min_{v\in V\setminus\{0\}}\frac{\langle A_\rmS v+2\lambda P_{\Pi\clU_M} v,v\rangle_{V',V}}{\norm{v}{H}}=\min_{v\in V\setminus\{0\}}\frac{\norm{v}{V}^2+2\lambda\norm{ P_{\Pi\clU_M} v}{H}^2}{\norm{v}{H}^2}.
\end{align}
Now we fix an auxiliary space of regular vector fields~$\widetilde\phi_i\in V$ so that the matrix~$[(\Pi\Phi_{i},\widetilde\phi_j)_H]$ is invertible, where the~$\Phi_i=\Phi_{n,i}^{[d]}$ are as in~\eqref{Act-intro}. This is possible by a continuity argument, because~$V$ is dense in~$H$ and the matrix~$[\clV_m]\coloneqq[(\Pi\Phi_i,\Pi\Phi_j)_H]$ is  invertible. The last statement is nontrivial, but can be proven by contradiction as follows. Let us assume that~$[\clV_m]$ is singular, then necessarily the $\Pi\Phi_i$ are linearly dependent, which implies that~$\xi\coloneqq\sum_{i=1}^m\sum_{n=1}^d c_{n,i}\Phi_{n,i}^{[d]}=\nabla g$ is a gradient vector field for some $g\in W^{1,2}(\Omega)$ and some constants~$c_{n,i}$.
This means that~$\curl\xi=0$. Clearly~$\frac{\partial\indf_{\omega_i}}{\partial x_j}$ is a distribution supported on the boundary~$\partial\omega_i$ of~$\omega_i$. Consequently, since the boundaries~$\partial\omega_i$ are pairwise disjoint, we have that $\curl\xi=0$ implies that, for each domain~$\omega_i$ we must have ~$\sum_{n=1}^d c_{n,i}\curl\Phi_{n,i}^{[d]}=0$. In particular, this implies that the family~$\{\frac{\partial\indf_{\omega_i}}{\partial x_k}\mid 1\le k\le d\}$ is linearly dependent. That is, for some constants~$\kappa_n$ the vector field~$\zeta\coloneqq(\kappa_1\indf{\omega_i},\dots,\kappa_d\indf{\omega_i})$ is divergence free. Then, for any smooth test function~$\varphi$ supported in~$\Omega$, we have~$0=\langle\diver\zeta,\varphi\rangle= -\langle\zeta,\nabla\varphi\rangle=-\int_{\partial\omega_i}(\kappa\cdot\bfn^{[i]})\varphi\rmd\partial\omega_i$, where $\bfn^{[i]}$ is the outward normal vector to~$\partial\omega_i$. Since~$\varphi$ is arbitrary, we obtain that~$\kappa\cdot\bfn^{[i]}$ vanishes on~$\partial\omega_i$. This is not possible for a fixed vector~$\kappa\in\bbR^d$ and for a bounded domain~$\omega_i\subset\bbR^d$, in particular, it is clearly not possible for rectangular/box domains as in Figs.~\ref{fig:act} and~\ref{fig:act-recttri}.
This is the sought  contradiction. 

Now, since the matrix~$[(\Pi\Phi_{i},\widetilde\phi_j)_H]$ is invertible, by using~\cite[Lem.~2.7]{KunRod19-cocv} we have the direct sum~$H=\widetilde \clU\oplus(\Pi\mathcal{U}_m)^\perp$ and the associated 
oblique projection~$P_{\widetilde \clU}^{(\Pi\mathcal{U}_m)^\perp}$ in~$H$ onto
\[
\widetilde \clU\coloneqq\linspan\{\widetilde\phi_j\mid 1\le j\le dm\}
\] along~$(\Pi\mathcal{U}_m)^\perp$ and the complementary projection~$P_{(\Pi\mathcal{U}_m)^\perp}^{\widetilde \clU}$ in~$H$ onto
$(\Pi\mathcal{U}_m)^\perp$ along~$\widetilde \clU$.
Analogously, since the matrix~$[\clV_m]=[(\Pi\Phi_i,\Pi\Phi_j)_H]$ is  invertible, we can write the orthogonal projection~$P_{\Pi\mathcal{U}_m}=P_{\Pi\mathcal{U}_m}^{(\Pi\mathcal{U}_m)^\perp}$ in~$H$ onto~$\Pi\clU_m$ as in~\eqref{Feed-orth}, \eqref{Feed-intro}, that is, $\Pi U_m^\diamond K_m^1 \coloneqq - P_{\Pi\mathcal{U}_m}$.
In particular, we have that any~$y\in H$ can be written in an unique way as
\[
y=\vartheta+\theta,\qquad\theta=P_{\widetilde \clU}^{(\Pi\mathcal{U}_m)^\perp}y\in \widetilde \clU,\quad\vartheta=P_{(\Pi\mathcal{U}_m)^\perp}^{\widetilde \clU}y\in(\Pi\mathcal{U}_m)^\perp.
\]
Furthermore, observe that, for~$y\in V\subset H$,
\begin{align*}
&\norm{\theta+\vartheta}{V}^2=\norm{\vartheta}{V}^2+\norm{\theta}{V}^2+2(\theta,\vartheta)_V\ge \norm{\vartheta}{V}^2+\norm{\theta}{V}^2-\frac12\norm{\vartheta}{V}^2-2\norm{\theta}{V}^2\ge \frac12\norm{\vartheta}{V}^2-\norm{\theta}{V}^2,
\end{align*}
which gives us, with~$\clY\coloneqq\widetilde \clU\times(\Pi\mathcal{U}_m)^\perp$,
\begin{align}\label{estEig1-1}
2\Lambda_1(\fkA_m^\lambda)&\ge\min_{\stackrel{y = \vartheta + \theta}{(\vartheta,\theta)\in \clY\setminus\{0\} }}\frac{\norm{\vartheta}{V}^2-2\norm{\theta}{V}^2+4\lambda \norm{ P_{\Pi\clU_m}y}{H}^2}{\norm{\vartheta+\theta}{H}^2}\\
&=\min_{(\vartheta,\theta)\in \clY \setminus\{0\}}\frac{\norm{\vartheta}{V}^2-2\norm{\theta}{V}^2+4\lambda \norm{ P_{\Pi\clU_m}\theta}{H}^2}{\norm{\vartheta+\theta}{H}^2}\\
&\ge\min_{(\vartheta,\theta)\in \clY\setminus\{0\} }\frac{\norm{\vartheta}{V}^2+(4\lambda C_{1m} -2C_{2m})\norm{\theta}{H}^2}{\norm{\vartheta+\theta}{H}^2}
\end{align}
with~$C_{1m}\coloneqq \min_{\theta\in\Pi \widetilde\clU_m}\frac{\norm{P_{\Pi\clU_m}\theta}{H}^2}{\norm{\theta}{H}^2}$ and~$C_{2m}\coloneqq \max_{\theta\in\Pi \widetilde\clU_m}\frac{\norm{\theta}{V}^2}{\norm{\theta}{H}^2}\ge1$. Note that~$\Pi \widetilde\clU_m$ is finite-dimensional and~$\theta\in \Pi \widetilde\clU_m$. Note also that~$C_{1m}>0$ because~$P_{\Pi\clU_m}\theta=0$ implies that~$\theta\in (\Pi\clU_m)^\perp$ and thus~$\theta\in\widetilde \clU\cap (\Pi\clU_m)^\perp=\{0\}$, that is, $P_{\Pi\clU_m}\theta=0$ if and only if~$\theta=0$.

Next, by Corollary~\ref{cor:poincare-AVH}, for~$h\le h^*=h^*(\zeta)$ and actuators associated to a partition~$\fkP_h$ of~$\Omega$ we find~$\norm{\vartheta}{V}\ge4\zeta\norm{\vartheta}{H}^2$. Next, we can choose~$\lambda_*=\lambda_*(\zeta, m)=\lambda_*(\zeta,h)$, such that~$\lambda_*\ge (2C_{1m})^{-1}(2\zeta+C_{2m}).$ Thus for~$\lambda\ge\lambda_*$, we find
\[
2\Lambda_1(\fkA_m^\lambda)\ge4\zeta\frac{\norm{\vartheta}{H}^2+\norm{\theta}{H}^2}{\norm{\vartheta+\theta}{H}^2}\ge2\zeta\frac{2\norm{\vartheta}{H}^2+2\norm{\theta}{H}^2}{\norm{\vartheta+\theta}{H}^2}\ge2\zeta,
\]
which finishes the proof of Theorem~\ref{thm:Eig1-AFeed}.\qed

	\section{Stabilizability}\label{sec:stabil}
Here, we prove the main stabilizability results of the manuscript in Theorem~\ref{th:main} below. Then, we show that the result announced in Theorem~\ref{th:main-intro} follows as a consequence. We assume that~$Y_0\in\bfH$.
\begin{theorem}\label{th:main}
 Let Assumptions~\ref{ass:Green}--\ref{ass:VcompH}, \ref{ass:target} and~\ref{ass:domain} hold true and let~$\zeta>0$. Then,  there exists~$h_*>0$ such that for every $h\le h_*$ and every partition~$\fkP_h=\{\clO_i\mid 1\le i\le m\}$ of~$\Omega\subset\bbR^d$ as in Assumption~\ref{ass:domain},
there exists~$\lambda_*>0$ such that for every~$\lambda\ge\lambda_*$, we have that every weak solution of system~\eqref{nse-dif-evol} satisfies
\begin{align*}
\norm{(y,\varphi)(t)}{\bfH,\beta}^2& \le \rme^{-\beta^{-1}\alpha(t-s)}\norm{(y,\varphi)(s)}{\bfH,\beta}^2,\qquad t\ge s\ge0
\end{align*}
with $\norm{(z,\phi)}{\bfH,\beta}$ as in~\eqref{normbetaH} and with~$\alpha$ as in~\eqref{fkDalpha}.
\end{theorem}

\begin{proof}
Testing the dynamics in~\eqref{nse-dif-evol} with~$2(y,\varphi)$ we find, with~$Y\coloneqq(y,\varphi)$,
\begin{align*}
\tfrac{\rmd}{\rmd t}\norm{Y}{\bfH,\beta}^2&=-2\norm{Y}{\bfV}^2-2(B(y,\widehat\bfy),y)_{V',V}
     +2(U_m^\diamond K_m y,y)_H\\
     &\le -2\norm{Y}{\bfV}^2+2C_\clN\widehat C_2\norm{y}{H}^\frac12\norm{y}{V}^\frac32-2\lambda \norm{P_{\clU_m} y}{H}^2  \le -\norm{Y}{\bfV}^2-2\lambda \norm{P_{\clU_m} y}{H}^2 +CC_\clN^4\widehat C_2^4\norm{y}{H}^2 \\
     &= -\langle A_\rmS y,y\rangle_{V',V} -\langle\fkD\varphi,\varphi\rangle_{\fkV',\fkV}-2\lambda \norm{P_{\clU_m} y}{H}^2 +CC_\clN^4\widehat C_2^4\norm{y}{H}^2 
\end{align*}
Now, by  Theorem~\ref{thm:Eig1-AFeed} it follows that, for any given~$\mu>0$ we can choose~$m$ and~$\lambda$ large enough so that
\begin{align*}
\tfrac{\rmd}{\rmd t}\norm{Y}{\bfH,\beta}^2& \le -\mu\norm{y}{H}^2 -\alpha\beta^{-1}\beta\norm{\varphi}{L^2}^2\le-\min\{\mu,\beta^{-1}\alpha\}\norm{Y}{\bfH,\beta}^2,
\end{align*}
with~$\alpha$ as in~\eqref{fkDalpha}. The rate~$\beta^{-1}\alpha$ can be achieved by choosing~$\mu\ge\beta^{-1}\alpha$.    
\end{proof}

\begin{proof}[Proof of Theorem~\ref{th:main-intro}]
It is sufficient to show that the construction illustrated in Fig.~\ref{fig:act} leads to the satisfiability of Assumption~\ref{ass:domain}. This will be shown in Section~\ref{satis-ass:domain}.
\end{proof}

\section{Satisfiability of the assumptions}\label{sec:satisfAss}

\subsection{Satisfiability of Assumptiom~\ref{ass:VnormW12}}
We show here that Assumptiom~\ref{ass:VnormW12} is satisfied in the cases~$\overline\fkL\in\{\nabla,\bfD,\curl\}$. 

In the case~$\overline\fkL=\nabla$ the result follows straightforwardly from the definition of the Sobolev $W^{1,2}$-norm. 
In the case~$\overline\fkL=\bfD$ we can use the Korn inequatily as in~\cite[Thm.~2, Ineq.~(2.21)]{FrattaSolombrino25}, \cite[Ineq.~(7)]{Nitsche81}. 
Finally, in the case~$\overline\fkL=\curl$ we can use~\cite[Appen.~I, Prop.~1.4]{Temam01}.

\subsection{Satisfiability of Assumptiom~\ref{ass:VcompH}}
By Assumption~\ref{ass:VnormW12} it follows that~$\bfV\xhookrightarrow{\rm c}\bfH$.

Now, let~$\bfw\coloneqq(w,\phi)\in\bfH$ satisfy $(\bfw,\bfv)_\bfH=0$ for all~$\bfv=(v,\psi)\in\bfV$.
Let also~$V_0\coloneqq\{h\in H\mid\; \overline{\fkL} h\in L^2\mbox{ and }h\rest{\partial\Omega}=0\}=\{h\in H\mid\; \nabla h\in L^2\mbox{ and }h\rest{\partial\Omega}=0\}$. Then, since~$V_0\subset V\times\{0\}\subset V$
we find that~$(w,v)_H=(\bfw,\bfv)_\bfH=0$ for all~$\bfv=(v,0)\in\bfV_0$.  Since~$V_0\xhookrightarrow{\rm d}H$ (e.g., see~\cite[Ch.~1, Sect.~1.4, Thms.~1.4 and 1.6]{Temam01}) we have that~$w=0$. Hence, $(\phi,\psi)_{L^2}=0$ for all~$\bfv=(v,\psi)\in\bfV$. 
Next, for any given~$\psi_*\in W^{\frac12,2}(\partial\Omega,\bfT\partial\Omega)$ there exists a vector field~$z_*\in V$ with~${z_*}\rest{\partial\Omega}=\psi_*$ (e.g., see~\cite[Ch.~1, Sect.~2.4, Thm.~2.4]{Temam01}). Thus,
$(z_*,\psi_*)\in\bfV$ and we can conclude that~$(\phi,\psi)_{L^2}=0$ for all~$\psi\in W^{\frac12,2}(\partial\Omega,\bfT\partial\Omega)$. Since~$W^{\frac12,2}(\partial\Omega,\bfT\partial\Omega)\xhookrightarrow{\rm d}L^{2}(\partial\Omega,\bfT\partial\Omega)$, we obtain that~$\phi=0$.
Therefore, necessarily $\bfw=0$ and we can conclude that~$\bfV\xhookrightarrow{\rm d}\bfH$.

\subsection{Satisfiability of Assumption~\ref{ass:target}}
It is common to assume an appropriate bound for the target trajectory; see~\cite[Sect.~2.2]{BarRodShi11}; even for steady states \cite[Eq.~(1.5)]{RaymThev10}. Note also that every smooth function solves the equations, for a suitable ``matching'' external forcing, so such regularity assumption on the target trajectory in Assumption~\ref{ass:target} is not empty.

\subsection{Satisfiability of Assumption~\ref{ass:domain}} \label{satis-ass:domain}

It is straightforward to see that Assumption~\ref{ass:domain} is satisfied by the partitions in Fig.~\ref{fig:act-recttri}, because the elements of the partition are the same, up to a translation and a rotation.

Next, we focus on the disk~$\bbD$ in Fig.~\ref{fig:act}, which we assume  to have radius~$1$, without loss of generality. 
We need to specify the details of the partition~$\bfP_M$ illustrated in Fig.~\ref{fig:act}. The partition includes a disk of radius~$R_{0M}\coloneqq\tfrac{3}{4M}$. For~$M>1$ the partition includes sectors~$\bbS_{M,k,j}$ of annulus~$\bbA_{M,k}$ of radius range~$[R_{0M}+(k-1)\rho_M,R_{0M}+k\rho_M]$, for~$k=1,...,M-1$, with~$\rho_M\coloneqq(1-R_{0M})(M-1)^{-1}$.
The number of sectors in each annulus is given by~$6k$ with angle range~$[(j-1)\theta_{M,k},j\theta_{M,k}]$, $\theta_{M,k}=\frac{2\pi}{6k}$, $1\le j\le 6k$.

We can see that there exists~$h=h(M)>0$ such that each sector in each annulus is contained in a translation of~$h\fkB_1$, and that we can take~$h\to0$  as~$M\to\infty$.

\subsubsection{Diffeomorphism to the unit box}
 Since the sectors in an annulus are the same up to a rotation, it is enough to consider the first ones~$\bbS_{M,k,1}$ with angle range~$[0,\theta_{M,k}]$. Observe also that, for a fixed~$M$, each sector of each annulus is diffeomorphic to the unit box~$\bbB\coloneqq[0,1]\times[0,1]$ through the following mapping, inspired by polar coordinates, 
\[
\Phi\colon (r,\theta)\mapsto ((a_1+a_2r)\cos(a_3\theta),(a_1+a_2r)\sin(a_3\theta)),
\]
with appropriate parameters~$a_i>0$, $1\le i\le 3$. Namely, with
\begin{subequations}\label{param-a}
\begin{align}
&(a_1,\,a_2,\,a_3)=a=a_{M,k}\coloneqq(R_{0M}+(k-1)\rho_M,\,\rho_M,\, \theta_{M,k}),
\intertext{where}
&R_{0M}=\tfrac{3}{4M},\qquad \rho_M=\tfrac{1-R_{0M}}{M-1},\qquad \theta_{M,k}=\frac{2\pi}{6k},\quad k\in\{1,\dots,M-1\}
\end{align}
\end{subequations}
are as above, we have that~$\Phi(\bbB)=\Phi_{M,k}(\bbB)=\bbS_{M,k,1}$.

Note that~$a_3\le\frac\pi3$ and~$\Phi$ is still an injective smooth function, when defined in~$[0,1]\times[-\tfrac32,\tfrac{3}{2}]$. Further, for~$M\ge2$,
\begin{align*}
(a_1+a_2r)>0&\quad\Longleftrightarrow\quad r>-\frac{a_1}{a_2}\quad\Longleftrightarrow\quad r>-\frac{R_{0M}+(k-1)\rho_M}{\rho_M}\\
&\quad\Longleftrightarrow\quad r>-\frac{3M-3}{4M-3}-(k-1)\quad\Longleftarrow\quad r>-\frac{3M-3}{4M-3}
\quad\Longleftarrow\quad r>-\frac{1}{3},
\end{align*}
hence~$\Phi=\Phi_{M,k}$ is still an injective smooth function, when defined in~$\widehat\bbB\coloneqq[-\frac13,1+\frac13]\times[-\tfrac32,\tfrac{3}{2}]$. We shall denote this extension of~$\Phi$ by~$\widehat\Phi\colon\widehat\bbB\to\bbR^2$.

The Jacobian matrix is given by
\[
J(\widehat\Phi)=\begin{bmatrix}
    a_2\cos(a_3\theta)&a_2\sin(a_3\theta))\\
   -a_3 (a_1+a_2r)\sin(a_3\theta)& \quad a_3(a_1+a_2r)\cos(a_3\theta)
\end{bmatrix}
\]
with determinant~$\det(J(\widehat\Phi))=a_2a_3(a_1+a_2r)>0$ and inverse
\[
J^{-1}(\widehat\Phi)
=\begin{bmatrix}
   a_2^{-1} \cos(a_3\theta)&- a_3^{-1} (a_1+a_2r)^{-1}\sin(a_3\theta)\\
   a_2^{-1}\sin(a_3\theta)& a_3^{-1} (a_1+a_2r)^{-1}\cos(a_3\theta)
\end{bmatrix}
\]

\subsubsection{Auxiliary estimates}
Let~$M\ge2$. We derive bounds for the entries of the Jacobian matrix and its inverse. We find
\begin{subequations}\label{aux-ineqJac}
\begin{align}
         M^{-1}\le a_2&=\rho_M\le (M-1)^{-1},
   \end{align}  
 and, after writing
 \begin{align}        
       a_3 (a_1+a_2r)&= \frac{\pi}{3}\rho_M +\frac{\pi}{3k}(R_{0M}+(r-1)\rho_M)\notag
  \intertext{we obtain}     
   a_3 (a_1+a_2r)&\le \frac{\pi}{3}\rho_M +\frac{\pi}{3k}(R_{0M}+\tfrac{1}{3}\rho_M)
   \le \frac{\pi}{3}\left(\rho_M +R_{0M}+\tfrac{1}{3}\rho_M\right)\notag\\& \le \frac{\pi}{3}\left(\tfrac43(M-1)^{-1} +\tfrac34M^{-1}\right)\le8\pi(M-1)^{-1},
  \end{align}  
  and
   \begin{align*}   
   a_3 (a_1+a_2r)&\ge \frac{\pi}{3}\rho_M +\frac{\pi}{3k}(R_{0M}-\tfrac{4}{3}\rho_M)
   \ge \frac{\pi}{3}\rho_M +\frac{\pi}{3k}((1+\tfrac{4}{3(M-1)})R_{0M}-\tfrac{4}{3}\tfrac{1}{M-1})\\
   &\hspace{-3em}= \frac{\pi}{3}\rho_M +\frac{\pi}{9(M-1)k}((3M+1)R_{0M}-4)= \frac{\pi}{3}\rho_M +\frac{\pi}{9(M-1)k}(\tfrac{9M+3}{4M}-4)\\
   &\hspace{-3em}= \frac{\pi}{3}\frac{4M-3}{M(M-1)} +\frac{\pi}{36M(M-1)k}(3-7M)= \frac{\pi}{3M(M-1)}((4-\tfrac{7}{12k})M-(3-\tfrac{1}{4k}))\\
   &\hspace{-3em}= \frac{\pi}{36M(M-1)}\frac{(48k-7)M-(9k-3)}k
   = \frac{\pi}{36M(M-1)}\frac{(48M-9)k+(3-7M)}k.
\end{align*}
Since the sequence~$k\mapsto\frac{(48M-9)k+(3-7M)}k$ is increasing, by taking the value at~$k=1$ we find
\begin{align}
   a_3 (a_1+a_2r)&\ge \frac{\pi(39M-6)}{36M(M-1)}\notag
\end{align}
and since the sequence~$M\mapsto\frac{39M-6}{M-1}$ is decreasing, by taking the limit as~$M\to\infty$ we obtain
\begin{align}
   a_3 (a_1+a_2r)&\ge \frac{39\pi}{36M}\ge \frac{3}{M}.
\end{align}
\end{subequations}
Gathering the above inequalities in~\eqref{aux-ineqJac} we find
 \begin{subequations}\label{aux-ineqJac2}
\begin{align}
&a_2\le (M-1)^{-1},\quad&& a_2^{-1}\le M,&&\\
   & a_3 (a_1+a_2r)\le 8\pi(M-1)^{-1},\quad && a_3^{-1} (a_1+a_2r)^{-1}\le 3^{-1}M,\quad &&-\frac13\le r \le\frac43.
   \end{align} 
\end{subequations}

\subsubsection{The Sobolev extension operators}

In the sectors~$\bbS_{M,k}$ of the partition~$\bfP_M$, we define the extension operator
\[
E_{M,k}\colon W^{1,2}(\bbS_{M,k})\to W^{1,2}(\bbR^2)
\]
as follows. We fix an extension operator
\[
E\in\clL(W^{1,2}(\bbB), W^{1,2}(\bbR^2))
\]
and a smooth cut-off function~$\Psi\colon  \bbR^2\to\bbR$ satisfying
\[
\Psi(x)=1\text{ for }x\in\bbB,\qquad \Psi(x)=0\text{ for }x\notin\widehat\bbB.
\]
Now the operator
\[
E_\Psi\in\clL(W^{1,2}(\bbB), W^{1,2}(\bbR^2)),\qquad (E_\Psi f)(x)\coloneqq \Psi(x)(E f)x)
\]
is still an extension operator with range consisting of functions supported in~$\widehat\bbB$. Therefore, hereafter we can simply assume that~$Ef$ is supported in~$\widehat\bbB$ and we will omit~$\Psi$.

The operator~$E_{M,k}$ then defined by
\[
(E_{M,k}f)(x)\coloneqq (E(f\circ\Phi_{M,k}))(\Phi_{M,k}^{-1}(x)).
\]

It remains to show that the operator norm of~$E_{M,k}$ is uniformly bounded on~$(M,k)$.
Let~$f\in W^{1,2}(\bbS_{M,k})$, $\bbS_{M,k}=\Phi(\bbB)$, $\Phi=\Phi_{M,k}$. By direct computations, with~$z=(r,\theta)\in \widehat\bbB$,
\begin{align*}
    \int_{\bbR^2} (E_{M,k}f)^2(x)\rmd x&=    \int_{\widehat\Phi(\widehat\bbB)} (E_{M,k}f)^2(x)\rmd x =\int_{\widehat\bbB} ((E(f\circ\widehat\Phi)(z)))^2\det (J(\widehat\Phi))\rmd z \\
    &\hspace{-3em}\le \widehat C_J\dnorm{E}{}^2\int_{\bbB} (f(\Phi(z)))^2\rmd z
\end{align*}
where
\[
\widehat C_J\coloneqq \sup_{z\in \widehat\bbB}\det (J(\widehat\Phi(z)))\quad\text{and}\quad\dnorm{E}{}\coloneqq\norm{E}{\clL(W^{1,2}(\bbB),W^{1,2}(\bbR^2))}.
\]
Then, continuing the computations,
\begin{align}
    \int_{\bbR^2} (E_{M,k}f)^2(x)\rmd x&\le \widehat C_J\dnorm{E}{}^2C_{J^{-1}}\int_{\bbB} (f(\Phi(z)))^2\det (J(\Phi))\rmd z\notag\\
    &=\widehat C_J\dnorm{E}{}^2 C_{J^{-1}}\int_{\Phi(\bbB)} f^2(x)\rmd x,\notag
\end{align}
where
\[
 C_{J^{-1}}\coloneqq \sup_{z\in \bbB}(\det (J(\Phi(z))))^{-1}.
\]

Recalling that~$\det(J(\widehat\Phi))=a_2a_3(a_1+a_2r)>0$, by~\eqref{aux-ineqJac2} we find that, since $M\ge2$,
\begin{equation}\label{bdd-CJhatJ-1}
    \widehat C_JC_{J^{-1}}\le8\pi(M-1)^{-2}3^{-1}M^2\le\frac{32\pi}{3}\le11\pi.
\end{equation}
Therefore, 
\begin{align}\label{ExtL2}
    \norm{E_{M,k}f}{L^2(\bbR^2)}^2\le 11\pi\dnorm{E}{}^2\norm{f}{L^2(\bbS_{M,k})}^2.
\end{align}

Next we consider the gradient. We start by observing that with~$x=\widehat\Phi(z)$, $x=(x_1,x_2)$, $z=(r,\theta)$, we have
\[
\nabla_z\coloneqq\begin{bmatrix}
    \tfrac{\partial}{\partial r}\\
    \tfrac{\partial}{\partial \theta}
\end{bmatrix}=J(\widehat\Phi)\begin{bmatrix}
    \tfrac{\partial}{\partial x_1}\\
    \tfrac{\partial}{\partial x_2}
\end{bmatrix}\eqqcolon J(\widehat\Phi)\nabla_x
\]
Then, we can proceed with direct computations 
\begin{align}
    \int_{\bbR^2} (\nabla_xE_{M,k}f)^2(x)\rmd x &=    \int_{\widehat\Phi(\widehat\bbB)} (\nabla_xE_{M,k}f)^2(x)\rmd x\notag\\     &\hspace{-8em} =\int_{\widehat\bbB} (J^{-1}(\widehat\Phi)\nabla_z(E(f\circ\widehat\Phi)(z)))^2\det (J(\widehat\Phi))\rmd z \le \widehat C_{[J^{-1}]}\widehat C_J\int_{\widehat\bbB} (\nabla_z(E(f\circ\widehat\Phi)(z)))^2\rmd z\notag\\     &\hspace{-8em} \le \widehat C_{[J^{-1}]}\widehat C_J\dnorm{E}{}\int_{\bbB} (\nabla_z(f(\Phi(z))))^2\rmd z\notag\\     &\hspace{-8em} \le \widehat C_{[J^{-1}]}\widehat C_J\dnorm{E}{}C_{[J]}C_{J^{-1}}\int_{\bbB} (J^{-1}(\Phi)\nabla_z(f(\Phi(z))))^2\det (J(\Phi))\rmd z\notag\\
    &\hspace{-8em} = \widehat C_{[J^{-1}]}\widehat C_J\dnorm{E}{}C_{[J]}C_{J^{-1}}\int_{\Phi(\bbB)} (\nabla_xf(x))^2\rmd x\notag
\end{align}
where
\[
\widehat C_{[J^{-1}]}\coloneqq \sup_{z\in \widehat\bbB}\norm{ J(\widehat\Phi(z)}{\ell^\infty}\quad\text{and}\quad C_{[J]}\coloneqq \sup_{z\in \bbB}\norm{ J(\Phi(z)}{\ell^\infty}.
\]

Using~\eqref{aux-ineqJac2} we find that, since $M\ge2$,
\begin{equation}\label{bdd-CJhatJ-2}
    \widehat C_{[J^{-1}]}C_{[J]}\le8\pi(M-1)^{-1}M\le16\pi.
\end{equation}
Therefore, from~\eqref{bdd-CJhatJ-1} and~\eqref{bdd-CJhatJ-2}, it follows
\begin{align}\label{ExtL2grad}
    \norm{\nabla E_{M,k}f}{L^2(\bbR^2)}^2\le 176\pi^2\dnorm{E}{}^2\norm{\nabla f}{L^2(\bbS_{M,k})}^2.
\end{align}
By combining~\eqref{ExtL2} and~\eqref{ExtL2grad}, we arrive at
\begin{align}\label{Ext-norm}
    \norm{E_{M,k}}{\clL(L^2(\bbS_{M,k}),L^2(\bbR^2))}^2\le 176\pi^2\dnorm{E}{}^2.
\end{align}
Hence, with a bound $C_E\coloneqq 176\pi^2\dnorm{E}{}^2$ independent of~$(M,k)$. We conclude that Assumption~\ref{ass:domain} is satisfied.

\subsubsection{Remarks}
 We have shown the satisfiability of Assumption~\ref{ass:domain} by using a diffeomorphism between the elements of the partition to a common box~$\bbB$. The same procedure can be applied to other domains. We could also consider other partitions, for example, triangulations as used in numerical simulations. In that case it will be natural to seek an isomorphism to a common reference triangle and the satisfiability of Assumption~\ref{ass:domain} means a certain uniformity/nondegeneracy of the triangular elements. This also means that likely, in practice, we could use the available triangulation software to seek appropriate partitions.
 
 An alternative to the proof of the satisfiability  of Assumption~\ref{ass:domain} could be to explore existing results in literature concerning uniform Sobolev extensions. For example, it would be enough to show that the elements of the partition are Jones $(\varepsilon,\delta)$-domains with a common pair$(\varepsilon,\delta)$; see~\cite[Ineqs.~(1.1)--(1.2) and Thm.~1]{Jones81}.

\bigskip
\noindent
{\bf Acknowledgments:}
The work of S. Rodrigues work is funded by national funds through the FCT -- Funda\c{c}\~{a}o para a Ci\^{e}ncia e a Tecnologia, I.P., under the scope of the projects UID/00297/2025 (\url{https://doi.org/10.54499/UID/00297/2025}) and UID/PRR/00297/2025 (\url{https://doi.org/10.54499/UID/PRR/00297/2025}) (Center for Mathematics and Applications -- NOVA Math)


	\bibliographystyle{plainurl}
	\bibliography{ObliqueDynBC}
	
\end{document}